\documentclass[11pt]{article}
\usepackage[latin1]{inputenc}
\usepackage{amssymb}
\usepackage{amsmath}
\usepackage{amsfonts}
\usepackage{amsthm}
\usepackage{enumerate}
\usepackage{bbm}

\title{Well-posedness results for the 3D Zakharov-Kuznetsov equation}
\date{}
\author{Francis Ribaud\thanks{Laboratoire d'Analyse et de Math\'ematiques Appliqu\'ees - Universit\'e Paris-Est - 5 Bd. Descartes - Champs-Sur-Marne, 77454 Marne-La-Vall\'ee Cedex 2 (FRANCE)} \and St\'ephane Vento\thanks{Laboratoire Analyse, G\'eom\'etrie et Applications - Universit\'e Paris 13 - Institut Galil\'ee, 99 avenue J.B. Cl\'ement,93430 Villetaneuse (FRANCE)}}

\numberwithin{equation}{section}

\newtheorem{theorem}{Theorem}[section]
\newtheorem{lemma}{Lemma}[section]
\newtheorem{proposition}{Proposition}[section]

\newcommand\re[1]{(\ref{#1})}
\def\R{\mathbb{R}}
\def\Z{\mathbb{Z}}

\def\S{\mathcal{S}}

\def\F{\mathcal{F}}

\def\eps{\varepsilon}

\def\supp{\mathop{\rm supp}\nolimits}
\def\sgn{\mathop{\rm sgn}\nolimits}

\newcommand\cro[1]{\langle #1 \rangle}

\begin{document}
\maketitle
\noindent {\bf Abstract.}\, We prove the local well-posedness of the three-dimensional Zakharov-Kuznetsov equation $\partial_tu+\Delta\partial_xu+ u\partial_xu=0$
in the Sobolev spaces $H^s(\R^3)$, $s>1$, as well as in the Besov space $B^{1,1}_2(\R^3)$. The proof is based on a sharp maximal function estimate in time-weighted spaces.
\\


\section{Introduction and main results}

In this paper we consider the local Cauchy problem for the three-dimensional Zakharov-Kuznetsov $(ZK)$ equation
\begin{equation}\label{zk}
\left\{\begin{array}{ll} u_t+\Delta u_x+uu_x=0,\\ u(0)=u_0,\end{array}\right.
\end{equation}
where $u=u(t,x,y,z)$, $u_0=u_0(x,y,z)$, $t  \in \R$ and $(x,y,z)\in \R^3$.

This equation was introduced by Zakharov and Kuznetsov in \cite{ZK} to describe the propagation of ionic-acoustic waves in magnetized plasma. The formal derivation of $(ZK)$ from the Euler-Poisson system with magnetic field can be found in \cite{LS}.

Clearly, the Zakharov-Kuznetsov equation can be considered as a multi-dimensional generalization of the well-known one-dimensional Korteweg-de Vries equation

\begin{equation}\label{kdv}
\left\{\begin{array}{ll} u_t+u_{xxx}+uu_x=0,\\ u(0)=u_0.\end{array}\right.
\end{equation}

We stress out the attention of the reader that contrary to some other generalizations of the Korteweg-de Vries equation (like the Kadomtsev-Petviasvili equations) the $(ZK)$ equation is not completely integrable and possesses only two invariant quantities by the flow. These two invariant quantities are the $L^2(\R^3)$ norm
$$
N(t)=\int_{\R^3}u^2(t,x,y,z) \,dxdydz \;,
$$
and the Hamiltonian
$$
H(t)=\frac{1}{2} \, \int_{\R^3} \left( (\nabla u(t,x,y,z))^2-\frac{u(t,x,y,z)^2}{3}\right) \, dxdydz\,.
$$
Hence, it is a  natural issue to study the Cauchy problem for the $(ZK)$ equation in the Sobolev space $H^1(\R^3)$ since any local well-posedness result in this space would provide global  well-posedness \footnote{Note that global existence occurs provided the local existence time only depends on the norm of the initial data in a suitable way. This is usually the case when solving the equation by a standard fixed point procedure.}.

Another difference with usual generalization of the KdV equation is that the resonant function associated to the $(ZK)$ equation seems too complex to develop a Bourgain approach, see \cite{LP}. Indeed, this function is defined in the hyperplane
$\bar{\xi_1}+\bar{\xi_2}+\bar{\xi_3}=0$ by
$$
h(\bar{\xi_1}, \bar{\xi_2}, \bar{\xi_3}) = \xi_1|\bar{\xi_1}|^2+\xi_2|\bar{\xi_2}|^2+\xi_3|\bar{\xi_3}|^2,\quad \bar{\xi_j}=(\xi_j, \eta_j,\mu_j)\in\R^3,
$$
and its zero set is not so easy to understand. This make again a   sharp difference with the Kadomtsev-Petviasvili equations where the use of the resonant function allows to derive local well-posedness (for the KP-II equation, see \cite{B}) or local ill-posedness (for the KP-I equation, see\cite{MST}). This explains why we follow the approach of Kenig, Ponce and Vega introduced in \cite{KPV3} for the study of the generalized KdV equation rather than a Bourgain approach.

In the the two-dimensional case,   Faminskii \cite{F} proved the local and global well-posedness of $(ZK)$ for initial data in $H^1(\R^2)$. This result was recently improved by Linares and Pastor  who obtained in \cite{LP} the local well-posedness  in $H^s(\R^2)$, $s>3/4$. The main tool to derive those results is the following $L^2_xL^{\infty}_{yT}$ linear estimate \cite{F},
$$
\forall s>3/4 , \quad \quad \|U(t)\varphi\|_{L^2_xL^{\infty}_{yT}}\leq C \| \varphi\|_{H^s(\R^2)}\;,
$$
where $U(t)\varphi$ denotes the free operator associated to the linear part of the $(ZK)$ equation. Note also  that in \cite{LPS}, Linares, Pastor and Saut recently obtained  local well-posedness results in some spaces which contains the one-dimensional solitary-waves of $(ZK)$ as well as  perturbations.

In the three-dimensional case, as far as we know, the only available result concerning the local well-posedness of $(ZK)$ in the usual Sobolev spaces goes back to Linares and Saut who proved in \cite{LS} the local well posedness in $H^s(\R^3$), $s>9/8$. In this paper we prove the following result,
\begin{theorem}\label{th-sobolev}
For any $s>1$ and $u_0\in H^s(\R^3)$, there exist $T>0$ and a unique solution $u$ of \re{zk} in
$$X_T^s\cap C_b([0,T], H^s(\R^3)).$$
Moreover, the flow-map $u_0\mapsto u$ is Lipschitz on every bounded set of $H^s(\R^3)$.
\end{theorem}
To derive our result, the main issue is to prove the local in time linear estimate
\begin{equation}\label{main-est}
\forall T<1, \quad \forall s>1, \quad  \|U(t)\varphi\|_{L^2_xL^{\infty}_{yzT}}\leq C \| \varphi\|_{H^s(\R^3)},
\end{equation}
where $U(t)\varphi$ denotes the free linear operator associated to the $(ZK)$ equation. Moreover, having a short look on the proof of Theorem \ref{th-sobolev}, we note that  any improvement of the linear estimate (\ref{main-est}) will immediately give the corresponding improvement for  Theorem \ref{th-sobolev}. Unfortunately we prove in Section \ref{sec-lin} that the $L^2_xL^{\infty}_{yzT}$ linear estimate (\ref{main-est}) fails when $s<1$.
Hence, this seems to indicate that the case $s=1$ could be critical for the well-posedness of the $(ZK)$ equation.

Concerning the "critical case" $s=1$, we have  unfortunately not been able to prove the local well-posedness in the natural energy space $H^1(\R^3)$. Nonetheless, working with initial data in the Besov space $B^ {1,1}_2(\R^ 3)$, we have the following
\begin{theorem}\label{th-besov}
For any $u_0\in B^{1,1}_2(\R^3)$, there exist $T>0$ and a unique solution $u$ of \re{zk} in
$$X_T\cap C_b([0,T], B^{1,1}_2(\R^3)).$$
Moreover, the flow-map $u_0\mapsto u$ is Lipschitz on every bounded set of $B^{1,1}_2(\R^3)$.
\end{theorem}
This result clearly improves Theorem \ref{th-sobolev} in view of the well-known embeddings
$$\forall s>1,\quad H^ s(\R^3)\hookrightarrow B^{1,1}_2(\R^3)\hookrightarrow H^1(\R^3).$$
To get Theorem \ref{th-besov}, our  main ingredient is to prove the unusual weighted-in-time linear estimate for phase localized functions
$$
\forall \alpha \geq 3/8, \quad \quad \|t^\alpha \Delta_k U(t)\varphi\|_{L^2_{x}L^\infty_{yzT}}\lesssim 2^k\|\Delta_k\varphi\|_{L^2}.
$$
This estimate, combined with the standard Kato smoothing estimate (\ref{smooth}), allows us to perform a fixed point argument on the Duhamel formulation of $(ZK)$.

This paper is organized as follows. In Section \ref{sec-not} we introduce our notations and define the resolution spaces. Section \ref{sec-lin} is devoted to estimates related to the linear part of the equation. Finally we prove the key bilinear estimates in Section \ref{sec-th}.

\section{Notation}\label{sec-not}
For $A,B>0$, $A\lesssim B$ means that there exists $c>0$ such that $A\leq cB$. When $c$ is a small constant we use $A\ll B$. We write $A\sim B$ to denote the statement that $A\lesssim B\lesssim A$.
For $u=u(t,x,y,z)\in\S'(\R^4)$, we denote by $\widehat{u}$ (or $\F u)$ its Fourier transform in space. The Fourier variables corresponding to a vector $\bar{x} = (x,y,z)$ will be denoted by $\bar{\xi} = (\xi, \eta, \mu)$. We consider the usual Lebesgue spaces $L^p$, $1\le p\le \infty$ and given a Banach space $X$ and a measurable function $u:\R\to X$, we define $\|u\|_{L^pX} = \left\|\|u(t)\|_X\right\|_{L^p}$. For $T>0$, we also set $L^p_T = L^p([0,T])$. Let us define the Japanese bracket $\cro{\bar{x}}=(1+|\bar{x}|^2)^{1/2}$ so that the standard non-homogeneous Sobolev spaces are endowed with the norm $\|f\|_{H^s}=\|\cro{\nabla}^sf\|_{L^2}$.

We use a Littlewood-Paley analysis. Let $p \in C^\infty_0(\R^d)$ be such that $p\geq 0$, $\supp p\subset B(0,2)$, $p\equiv 1$ on $B(0,1)$. We define next $p_k(\bar{\xi})=p(\bar{\xi}/2^k)$ for $k\ge 0$.
We set $\delta(\bar{\xi})=p(\bar{\xi}/2)-p(\bar{\xi})$ and $\delta_k(\bar{\xi})=\delta(\bar{\xi}/2^k)$ for any $k\in\Z$, and define the operators
$P_k$ ($k\ge 0$) and $\Delta_k$ ($k\in\Z$) by $\F(P_ku)=p_k \widehat{u}$ and $\F(\Delta_k)=\delta_k\widehat{u}$. When $d=3$, we introduce the operators $P_k^x$, $P_k^y$, $P_k^z$, and $\Delta_k^x$, $\Delta_k^y$, $\Delta_k^z$ defined by
$$\left\{\begin{array}{lll}P_k^xu(x,y,z)=\F^{-1}(p_k(\xi)\widehat{u}(\xi,\eta,\mu)),\\ P_k^yu(x,y,z)=\F^{-1}(p_k(\eta)\widehat{u}(\xi,\eta,\mu)),\\ P_k^zu(x,y,z)=\F^{-1}(p_k(\mu)\widehat{u}(\xi,\eta,\mu))\end{array}\right.$$
and
$$\left\{\begin{array}{lll}\Delta_k^xu(x,y,z)=\F^{-1}(\delta_k(\xi)\widehat{u}(\xi,\eta,\mu)),\\ \Delta_k^yu(x,y,z)=\F^{-1}(\delta_k(\eta)\widehat{u}(\xi,\eta,\mu)),\\ \Delta_k^zu(x,y,z)=\F^{-1}(\delta_k(\mu)\widehat{u}(\xi,\eta,\mu))\end{array}\right.$$
Furthermore we define more general projections $P_{\lesssim k}=\sum_{j: 2^j\lesssim 2^k}P_j$, $\Delta_{\gg j}^x=\sum_{j: 2^j\gg 2k}\Delta_j^x$ etc.
For future considerations, note that for $u\in\S'(\R^3)$ and $p\in[1,\infty]$ we have
\begin{equation}
\|\Delta_ku\|_{L^p}\lesssim \|\Delta_k^xP_{\lesssim k}^yP_{\lesssim k}^zu\|_{L^p}+\|P_{\lesssim k}^x\Delta_k^yP_{\lesssim k}^zu\|_{L^p}+\|P_{\lesssim k}^xP_{\lesssim k}^y\Delta_k^zu\|_{L^p}.
\end{equation}

With these notations, it is well known that an equivalent norm on $H^s(\R^d)$ is given by
$$\|u\|_{H^s}\sim \|P_0u\|_{L^2}+\left(\sum_{k\ge 0}2^{2sk}\|\Delta_ku\|_{L^2}^2\right)^{1/2}.$$
For $s\in\R$, the Besov space $B^{s,1}_2(\R^d)$ denotes the completion of $\S(\R^d)$ with respect to the norm
$$\|u\|_{B^{s,1}_2} = \|P_0u\|_{L^2}+\sum_{k\ge 0}2^{sk}\|\Delta_ku\|_{L^2}.$$

\section{Linear estimates}\label{sec-lin}
Consider the linear ZK equation
\begin{equation}\label{linearZK}u_t+\Delta u_x=0,\quad t\in\R, x\in\R^3.\end{equation}
Let $\omega(\bar{\xi})=\xi(\xi^2+\eta^2+\mu^2)$ and $U(t)=\F^{-1}e^{it\omega(\bar{\xi})}\F$, be the associated linear operator.

First we prove a standard "Kato smoothing" estimate for the free evolution of \re{linearZK}.
\begin{proposition}\label{prop-smooth}
For any $\varphi\in L^2(\R^3)$, it holds that
\begin{equation}\label{smooth}\|\nabla U(t)\varphi\|_{L^\infty_{x}L^2_{yzt}}\lesssim \|\varphi\|_{L^2}.\end{equation}
\end{proposition}

\begin{proof}
The proof is modeled on the corresponding result for the KdV equation \cite{KPV1} (see also \cite{KPV2} and \cite{F} for the two-dimensional case). We perform the change of variables $\theta = h(\xi) = \omega(\bar{\xi})$ and obtain
\begin{align*}
 U(t)\varphi(\bar{x}) &= \int_{\R^3}e^{i(\bar{x}\bar{\xi}+t\omega(\bar{\xi})}\widehat{\varphi}(\bar{\xi})d\bar{\xi}\\
&= \F^{-1}_{\theta\eta\mu}\left(e^{ixh(\theta)}(h^{-1})'(\theta)\widehat{\varphi}(h(\theta),\eta,\mu)\right)(t,y,z).
\end{align*}
Therefore, applying Plancherel theorem and returning to the $\xi$-variable yield
\begin{align*}
 \|U(t)\varphi(x)\|_{L^2_{yzt}} &= \|(h^{-1})'(\theta)\widehat{\varphi}(h(\theta),\eta,\mu)\|_{L^2_{\theta\eta\mu}}\\
&= \||h'(\xi)|^{-1/2}\widehat{\varphi}(\xi,\eta,\mu)\|_{L^2_{\xi\eta\mu}}\\
&\sim \|\nabla^{-1}\varphi\|_{L^2}.
\end{align*}
\end{proof}

We will use in a crucial way the following maximal estimate for the free evolution when acting on phase localized functions.

\begin{proposition}\label{prop-max}
Let $0<T<1$ and $\alpha\ge 3/8$.
\begin{enumerate}
 \item For all $\varphi\in \S(\R^3)$ and $k\ge 0$, we have
\begin{equation}\label{maxEst3d}\|t^\alpha \Delta_k U(t)\varphi\|_{L^2_{x}L^\infty_{yzT}}\lesssim 2^k\|\Delta_k\varphi\|_{L^2}.\end{equation}
\item For any $\varphi\in\S(\R^3)$, it holds that
\begin{equation}
 \|P_0 U(t)\varphi\|_{L^2_xL^\infty_{yzT}}\lesssim \|P_0\varphi\|_{L^2}.
\end{equation}

\end{enumerate}
\end{proposition}

Before proving Proposition \ref{prop-max}, we first show some estimates related to the oscillatory integrals:
$$I_0(t,\bar{x}) = \int_{\R^3}e^{i(\bar{x}\bar{\xi}+t\omega(\bar{\xi}))}p_0(\bar{\xi})d\bar{\xi},$$
and for $k\ge 1$,
$$I_k(t,\bar{x}) = \int_{\R^3}e^{i(\bar{x}\bar{\xi}+t\omega(\bar{\xi}))}\psi_1(\xi)\psi_2(\eta)\psi_3(\mu)d\bar{\xi}$$
where $(\psi_1, \psi_2, \psi_3) = (\delta_k, p_k, p_k)$, $(p_k, \delta_k, p_k)$ or $(p_k, p_k, \delta_k)$.

\begin{lemma}\label{lem-i0}
 $$\|I_0\|_{L^1_xL^\infty_{yzT}}\lesssim 1.$$
\end{lemma}

\begin{proof}
Since $|I_0|\lesssim 1$, it is clear that $\|I_0\|_{L^1_xL^\infty_{yzT}}\lesssim 1$ if we are in the region $|x|\lesssim 1$.
Thus we may assume that $|x|\gg 1$. Define the phase function $\varphi_1$ by $\varphi_1(\xi) = t\omega(\bar{\xi})+x\xi$ so that $\varphi_1'(\xi) = t(3\xi^2+\eta^2+\mu^2)+x$. On the support of $p_0$, we have $|\varphi_1'|\gtrsim |x|$.
Thus, two integrations by parts yield the estimate
$$\left|\int_\R e^{\phi_1(\xi)}p_0(\bar{\xi})d\xi\right| \lesssim
\int_\R \left|\frac{p_{0\xi\xi}}{\varphi_1''}\right| + \left|\frac{p_{0\xi}\varphi_1''^2}{\varphi_1'^3}\right| + \left|\frac{p_0\varphi_1'''}{\varphi_1'^3}\right| + \left|\frac{p_0\varphi_1''^2}{\varphi_1'^4}\right| \lesssim |x|^{-2}.
$$
It follows that $|I_0|\lesssim |x|^{-2}$, which implies that $\|I_0\|_{L^1_xL^\infty_{yzT}}\lesssim 1$ as required.
\end{proof}

\begin{lemma}\label{lem-ik} For any $\alpha\ge 3/8$ and $k\ge 0$, it holds that
 $$\|t^{2\alpha} I_k\|_{L^1_xL^\infty_{yzT}}\lesssim 2^{2k}.$$
\end{lemma}

\begin{proof}
We split  $I_k$ into
\begin{align}\notag
I_k &= \int_{\R^3}e^{i(\bar{x}\bar{\xi}+t\omega(\bar{\xi}))}\psi_1(\xi)(1-p_0(\xi))\psi_2(\eta)\psi_3(\mu)d\bar{\xi} \\ \notag &\quad + \int_{\R^3}e^{i(\bar{x}\bar{\xi}+t\omega(\bar{\xi}))}p_0(\xi)\psi_2(\eta)\psi_3(\mu)d\bar{\xi}\\
&:= I_k^1+I_k^2.
\label{eq-iksplit}
\end{align}

\begin{itemize}
\item
Estimate for $I_k^1$.\\
Since  we have $|\xi|\gtrsim 1$, a rough estimate for $I_k^1$ yields $|I_k^1|\lesssim 2^{3k}$, which gives the desired bound in the region where $|x|\le 2^{-k}$.
Therefore we may assume $|x|\ge 2^{-k}$. If we have either $|x|\ll t2^{2k}$ or $|x|\gg t2^{2k}$, then using that $|\omega'|\sim 2^{2k}$ we infer $|\varphi_1'|\gtrsim \max(|x|, t2^{2k})$ where
$\varphi_1$ is the phase function \begin{equation}\label{phi1}\varphi_1(\xi) = t\omega(\bar{\xi})+x\xi.\end{equation} Integrating by parts twice with respect to $\xi$ we deduce
$$\left|\int_\R e^{i\varphi_1}\psi_1(1-p_0)\right|\lesssim \max(|x|, t2^{2k})^{-2}.$$
It follows that $|I_k^1|\lesssim 2^{2k}\max(|x|, t2^{2k})^{-2}$ and next
$$t^{1/2}|I_k^1|\lesssim |x|^{-3/2}2^k,$$
which is acceptable since we integrate in the region $|x|\ge 2^{-k}$. Now we consider the case $|x|\sim t2^{2k}$. Using that
$$\int_{\R^2}e^{it\xi(\eta^2+\mu^2)+iy\eta+iz\mu}d\eta d\mu = \frac{\pi}{t|\xi|}e^{-i\frac{y^2+z^2}{4t\xi}}e^{i\frac \pi 2 \sgn(\xi)},$$
we may rewrite $I_k^1$ as
\begin{equation}\label{ik}I_k^1 = \int_{\R^2}\check{\psi_2}(y-u)\check{\psi_3}(z-v)\left(\int_\R\frac{\pi i}{t\xi}e^{i\varphi_2(\xi)}\psi_1(\xi)(1-p_0(\xi))d\xi\right)dudv\end{equation}
where we set $\varphi_2(\xi)=t\xi^3+x\xi-\frac{u^2+v^2}{4t\xi}$. Since $|\varphi_2'''|\ge 6t$ on the support of $1-p_0$, Van der Corput's lemma implies that
$$\left|\int_\R\frac{\pi i}{t\xi}e^{i\varphi_2(\xi)}\psi_1(\xi)(1-p_0(\xi))d\xi\right|\lesssim t^{-4/3}.$$
It follows that $|I_k^1|\lesssim t^{-4/3}$ and
$$t^{2\alpha}|I_k^1|\lesssim t^{2\alpha-4/3}\lesssim |x|^{2\alpha-4/3}2^{-2k(2\alpha-4/3)}.$$
Hence we obtain for $\alpha> 1/6$ and $|x|\ge 2^{-k}$ that
$$\|t^{2\alpha}I_k^1\|_{L^1_xL^\infty_{yzT}}\lesssim 2^{2k(2\alpha-1/3)}2^{-2k(2\alpha-4/3)}\sim 2^{2k}.$$

\item
Estimate for $I_k^2$.
\\We treat now the low frequencies term $I_k^2$. The case $|x|\lesssim 1$ is easily handled since we have the rough estimate $|I_k^2|\lesssim 2^{2k}$. Thus we only need to consider the region where $|x|\gg 1$. In the domain $|x|\ll t2^{2k}$ or $|x|\gg t2^{2k}$, we have $|\varphi_1'|\gtrsim \max(|x|, t2^{2k})$ where $\varphi_1$ is defined in (\ref{phi1}), and thus
$$\left|\int_\R e^{i\varphi_1}p_0\right|\lesssim |x|^{-2}.$$
It follows that $|I_k^2|\lesssim 2^{2k}|x|^{-2}$ and
$$\|I_k^2\|_{L^1_xL^\infty_{yzT}}\lesssim 2^{2k}.$$
Now we consider the most delicate case $|x|\sim t2^{2k}$ and rewrite $I_k^2$ as in (\ref{ik}) where $\psi_1(1-p_0)$ is replaced with $p_0$. Let us split $p_0$ as $$p_0=p_{-2k}+\sum_{j=-2k}^0\delta_j.$$
The part $p_{-2k}$ is straightforward since $|I_k^2|\lesssim 1$ and we get from $|x|\sim t2^{2k}\lesssim 2^{2k}$ that $\|I_k^2\|_{L^1_xL^\infty_{yzT}}\lesssim 2^{2k}$. Thus we reduce to estimate
$$I_{k,j}^2 = \int_{\R^2}\check{\psi_2}(y-u)\check{\psi_3}(z-v)\left(\int_\R\frac{\pi i}{t\xi}e^{i\varphi_2(\xi)}\delta_j(\xi)d\xi\right)dudv$$
for $j=-2k,\ldots, 0$. First consider the case $|x|\ll t^{-1}2^{-2j}(u^2+v^2)$ or $|x|\gg t^{-1}2^{-2j}(u^2+v^2)$. Since $\varphi_2'(\xi)=3t\xi^2+x+\frac{u^2+v^2}{4t\xi^2}$, we have $|\varphi_2'|\gtrsim\max(|x|, t^{-1}2^{-2j}(u^2+v^2))$ and an application of the Van der Corput's lemma yields
\begin{align*}
|I_{k,j}^2| &\lesssim \int_{\R^2}|\check{\psi_2}(y-u)\check{\psi_3}(z-v)|t^{-1}|x|^{-3/4}(t^{-1}2^{-2j}(u^2+v^2))^{-1/4}2^{-j}dudv\\
&\lesssim |x|^{-3/4}t^{-3/4}2^{-j/2}\int_\R \frac{|\check{\psi_2}(y-u)|}{|u|^{1/2}}du.
\end{align*}
On the other hand, the change of variables $v=2^ku$ leads to
\begin{align*}
\int_\R \frac{|\check{\psi_2}(y-u)|}{|u|^{1/2}}du &= 2^k\int_\R\frac{|\check{\psi}(2^ky-2^ku)|}{|u|^{1/2}}du\\
&= 2^{k/2}\int_\R \frac{|\check{\psi}(2^ky-v)|}{|v|^{1/2}}dv\\
&\lesssim 2^{k/2}\int_{|v|\le 1}\frac{dv}{|v|^{1/2}}+2^{k/2}\int_{|v|\ge 1}|\check{\psi}(2^ky-v)|dv\\
&\lesssim 2^{k/2}.
\end{align*}
Consequently, it is deduced that $t^{3/4}|I_{k,j}^2|\lesssim |x|^{-3/4}2^{-j/2}2^{k/2}$ and
$$\|t^{3/4}I_k^2\|_{L^1_xL^\infty_{yzT}}\lesssim \sum_{j=-2k}^0 2^{-j/2}2^{k/2}\int_{|x|\lesssim 2^{2k}}\frac{dx}{|x|^{3/4}}\lesssim 2^{2k}.$$
Finally assume that $|x|\sim t^{-1}2^{-2j}(u^2+v^2)$ so that $|\varphi_2''|\gtrsim t2^{2k-j}$. Then we get from Van der Corput's lemma that
$$|I_{k,j}^2|\lesssim (t2^{2k-j})^{-1/2}t^{-1}2^{-j}\sim t^{-3/2}2^{-k}2^{-j/2}.$$
Hence, we obtain
$$t^{2\alpha}|I_k^2|\lesssim t^{2\alpha-3/2}2^{-k}\sum_{j=-2k}^02^{j/2}\lesssim |x|^{2\alpha-3/2}2^{-2k(2\alpha-3/2)},$$
which is acceptable as soon as $\alpha>1/4$. This concludes the proof of Lemma \ref{lem-ik}.
\end{itemize}
\end{proof}

We are now in a position to prove Proposition \ref{prop-max}.

\begin{proof}[Proof of Proposition \ref{prop-max}]
 Let $k\ge 0$ and $U_k(t) = \Delta_kU(t)$. The proof will follow from a slight modification of the usual $AA^*$ argument.
Let us define the operator $A_k : L^1_TL^2_{\bar{x}}\to L^2$ by
$$A_kg = \int_0^Tt^\alpha U_k(-t)g(t)dt.$$
We easily check that $A_k^*h(t) = t^\alpha U_k(t)h$ for $h\in L^2(\R^3)$ and moreover,
$$A_k^*A_kg(t) = \int_0^T(tt')^\alpha U_k(t-t')g(t')dt'.$$
The previous integrand can be estimated by
\begin{align*}
|(tt')^\alpha U_k(t-t')g(t',\bar{x})| &\lesssim \big||t-t'|^{2\alpha} U_k(t-t')g(t',\bar{x})\big| + \big||t+t'|^{2\alpha} U_k(t-t')g(t',\bar{x})\big|\\
& := I+II.
\end{align*}
Using the Young inequality, the first term is bounded by
\begin{align*}
 I &\lesssim \left|\big(|t-t'|^{2\alpha}I_k(t-t') \ast_{\bar{x}}g(t')\big)(\bar{x})\right|\\
& \lesssim \left((|t-t'|^{2\alpha}\|I_k(t-t')\|_{L^\infty_{yz}})\ast_x \|g(t')\|_{L^1_{yz}}\right)(x).
\end{align*}
Integrating this into $t'\in[0,T]$ and taking the $L^2_xL^\infty_{yzT}$ norm, this leads to
\begin{align}\notag
 \left\|\int_0^T \big||t-t'|^{2\alpha} U_k(t-t')g(t',\bar{x})\big| dt'\right\|_{L^2_xL^\infty_{yzT}} &\lesssim
\left\|\|t^{2\alpha} I_k\|_{L^\infty_{yzT}} \ast_x \|g(t)\|_{L^1_{yzT}}\right\|_{L^2_x}\\
\label{est-aa1} &\lesssim \|t^{2\alpha}I_k\|_{L^1_xL^\infty_{yzT}}\|g\|_{L^2_xL^1_{yzT}}.
\end{align}
To estimate $II$, we introduce $\check{g}(t,\bar{x}) = g(t,-\bar{x})$ and notice that
$$U_k(t-t')g(t',\bar{x}) = U_k(t+t')\check{g}(t', -\bar{x}).$$
We infer that
$$
II \lesssim \left((|t+t'|^{2\alpha})\|I_k(t+t')\|_{L^\infty_{yz}}) \ast_x \|\check{g}(t')\|_{L^1_{yz}}\right)(x)
$$
and we get
\begin{align}\notag
\left\|\int_0^T\big||t+t'|^{2\alpha}U_k(t-t')g(t',\bar{x})dt'\right\|_{L^2_xL^\infty_{yzT}} &\lesssim \left\|\|t^{2\alpha}I_k\|_{L^\infty_{yzT}}\ast_x\|\check{g}\|_{L^1_{yzT}}\right\|_{L^2_x} \\
\label{est-aa2} & \lesssim \|t^{2\alpha}I_k\|_{L^1_xL^\infty_{yzT}} \|g\|_{L^2_xL^\infty_{yzT}}.
\end{align}
Combining estimates (\ref{est-aa1})-(\ref{est-aa2}) and Lemma \ref{lem-ik} we deduce
$$\|A_k^*A_kg\|_{L^2_xL^\infty_{yzT}} \lesssim \|t^{2\alpha}I_k\|_{L^1_xL^\infty_{yzT}} \|g\|_{L^2_xL^{yzT}} \lesssim 2^{2k}\|g\|_{L^2_xL^\infty_{yzT}}.$$
The usual algebraic lemma (see Lemma 2.1 in \cite{GV}) applies and yields the first estimate in Proposition \ref{prop-max}. The second one is obtained by following the same lines and using Lemma \ref{lem-i0}.
\end{proof}

In order to get the desired bounds for data in $H^s(\R^3)$, $s>1$, we will use the following estimate.
\begin{proposition}\label{prop-maxHs}
For $0<T<1$, $s>1$ and any $\varphi\in\S(\R^3)$, it holds that
\begin{equation}\label{est-linhs}
 \|U(t)\varphi\|_{L^2_xL^\infty_{yzT}}\lesssim \|\varphi\|_{H^s}.
\end{equation}
\end{proposition}
This proposition is a direct consequence of Lemma \ref{lem-i0} together with the following result.
\begin{lemma} For any $\eps>0$ and $k\ge 0$, it holds that
 $$\|I_k\|_{L^1_xL^\infty_{yzT}}\lesssim 2^{(2+\eps)k}.$$
\end{lemma}
\begin{proof}
Setting
$$I_{i,j,k}(t,\bar{x}) = \int_{\R^3}e^{i(\bar{x}\bar{\xi}+t\omega(\bar{\xi}))}\delta_i(\xi)\delta_j(\eta)\delta_k(\mu)d\bar{\xi},$$
we see that from the estimate
$$\|\Delta_n u\|_{L^1_xL^\infty_{yzT}}\lesssim \sum_{i,j,k: 2^i,2^j,2^k\lesssim 2^n}\|\Delta_i^x\Delta_j^y\Delta_k^zu\|_{L^1_xL^\infty_{yzT}},$$
it suffices to show that
\begin{equation}\label{est-Iijk}
\|I_{i,j,k}\|_{L^1_xL^\infty_{yzT}}\lesssim (1+M)2^{2M}
\end{equation}
for all $i,j,k\in\Z$ and with $M=\max(i,j,k)\ge 0$. From the straightforward bound
$|I_{i,j,k}|\lesssim 2^{i+j+k}$, we see that we may assume $|x|\ge 2^{-m}$ where $m=\min(i,j,k)$. In the region $|x|\ll t2^{2M}$ or $|x|\gg t2^{2M}$, the phase function
$\varphi_1$ defined in \re{phi1} satisfies $|\varphi_1'|\gtrsim\max(|x|, t2^{2M})$ and integrations by parts lead to
$$\left|\int_\R e^{i\varphi_1}\delta_i\right|\lesssim 2^{-i}|x|^{-2}.$$
Therefore we get $|I_{i,j,k}|\lesssim 2^{j+k-i}|x|^{-2}$ and then $\|I_{i,j,k}\|_{L^1_xL^\infty_{yzT}}\lesssim 2^{j+k-i+m}\lesssim 2^{2M}$.
Finally consider the case $|x|\sim t2^{2M}$. Rewriting $I_{i,k,j}$ as
$$I_{i,j,k} = \int_{\R^2}\check{\delta_j}(y-u)\check{\delta_k}(z-v)\left(\int_\R \frac{\pi i}{t\xi}e^{i\varphi_2(\xi)}\delta_i(\xi)d\xi\right)dudv,$$
we immediately obtain $|I_{i,j,k}|\lesssim t^{-1}\lesssim 2^{2M}|x|^{-1}$ and thus $\|I_{i,j,k}\|_{L^1_xL^\infty_{yzT}}\lesssim M2^{2M}$, as desired.

\end{proof}

Up to the end point $s=1$, we show next that estimate \re{est-linhs} is sharp.
\begin{proposition}
Suppose that for any $\varphi\in H^s(\R^3)$ we have
\begin{equation}\label{est-sharp}
\|t^\alpha U(t)\varphi\|_{L^2_xL^\infty_{yzT}}\lesssim \|\varphi\|_{H^s},
\end{equation}
for some $\alpha\ge 0$. Then it must be the case that $s\ge 1$.
\end{proposition}
\begin{proof}
 Let us define the smooth functions $\varphi_k$ through their Fourier transforms by $$\hat{\varphi}(\bar{\xi}) = \delta_{-2k}(\xi)\delta_k(\eta)\delta_k(\mu)$$
for $k\ge 0$. Then it is easy to see that
\begin{equation}\label{est-phik}
 \|\varphi_k\|_{H^s}\sim 2^{ks}.
\end{equation}
On the other hand, for $\eps>0$ small enough, we set $t=\eps$ and $y=z=\eps 2^{-k}$ so that $|y\eta+z\mu+t\omega(\bar{\xi})|\lesssim \eps$. Choosing $|x|\ll 2^{2k}$, we obtain the lower bound
\begin{align*}
 |U(t)\varphi_k(\bar{x})| &= \left|\int_{\R^3}e^{i(\bar{x}\bar{\xi}+t\omega(\bar{\xi}))}\varphi_k(\bar{\xi})d\bar{\xi}\right|\\
&= \Big|\int_{\R^3}[e^{i(y\eta+z\mu+t\omega(\bar{\xi}))}-1]e^{ix\xi}\delta_{-2k}(\xi)\delta_k(\eta)\delta_k(\mu)d\bar{\xi}\\
&\quad +\int_{\R^3}e^{ix\xi}\delta_{-2k}(\xi)\delta_k(\eta)\delta_k(\mu)d\bar{\xi}\Big|\\
&\gtrsim 1.
\end{align*}
It follows that $\|t^\alpha U(t)\varphi_k\|_{L^2_xL^\infty_{yzT}}\gtrsim 2^k$ where the implicit constant does not depend on $k$. Therefore, \re{est-sharp} and \re{est-phik} imply
$$2^k\lesssim 2^{ks}.$$
From this, we get for large $k$ that $s\ge 1$.
\end{proof}

In the sequel of this section we prove retarded linear estimates which will be used later to perform the fixed point argument.
\begin{proposition}
 Let $f\in\S(\R^4)$. Then we have
\begin{equation}
\label{est-nhgroup} \left\|\nabla\int_0^tU(t-t')f(t')dt'\right\|_{L^\infty_TL^2_{\bar{x}}}\lesssim \|f\|_{L^1_xL^2_{yzT}}.
\end{equation}
\end{proposition}
\begin{proof}
 The dual estimate of \re{smooth} reads
\begin{equation}\label{dual-smooth}\left\|\int_{-\infty}^\infty U(-t')\nabla f(t')dt'\right\|_{L^2}\lesssim \|f\|_{L^1_xL^2_{yzT}}.\end{equation}
Noticing that $U(t)$ is a unitary group on $L^2(\R^3)$ we obtain for any fixed $t$,
\begin{equation}\label{est-dual}
\left\|\int_{-\infty}^\infty U(t-t')\nabla f(t')dt'\right\|_{L^2_{\bar{x}}}\lesssim \|f\|_{L^1_xL^2_{yzT}}.
\end{equation}
To conclude we substitute in \re{est-dual} $f(t')$ by $\chi_{[0,t]}(t')f(t')$ and then take the supremum in time in the left-hand side of the resulting inequality.
\end{proof}

\begin{proposition}
 Let $f\in\S(\R^4)$. Then we have
\begin{equation}
\label{est-nhsmooth} \left\|\nabla^2\int_0^tU(t-t')f(t')dt'\right\|_{L^\infty_xL^2_{yzT}}\lesssim \|f\|_{L^1_xL^2_{yzT}}.
\end{equation}
\end{proposition}

\begin{proof}
 We first observe that for any $g:\R\to\R$,
$$\int_0^tg(t')dt' = \frac 12\int_\R g(t')\sgn(t-t')dt' + \frac 12\int_\R g(t')\sgn(t')dt',$$
where $\sgn(\cdot)$ denotes the sign function. Consequently we have
\begin{align}
\notag \nabla^2R(t,\bar{x}) &= \nabla^2\int_0^tU(t-t')f(t')dt'\\
\notag&= \frac 12\nabla^2\int_\R U(t-t')f(t')\sgn(t-t')dt' + \frac 12\nabla^2\int_\R U(t-t')f(t')\sgn(t')dt'\\
\label{dec-R}&= \frac 12\nabla^2R_1(t,\bar{x})+\frac 12\nabla^2R_2(t,\bar{x}).
\end{align}
Taking the inverse space-time Fourier transform, it is clear that
$$R_1 = \F^{-1}_{\tau \bar{\xi}}\left(\widehat{\sgn}(\tau-\omega(\bar{\xi}))\hat{f}(\tau,\bar{\xi})\right).$$
Hence we get by Plancherel theorem
\begin{align}
\notag\|\nabla^2 R_1\|_{L^2_{yzt}} &= \left\|\F^{-1}_\xi\left(|\bar{\xi}|^2\widehat{\sgn}(\tau-\omega(\bar{\xi}))\hat{f}(\tau,\bar{\xi})\right)\right\|_{L^2_{\eta\mu\tau}}\\
\label{est-R11}& = \|K(\tau,x,\eta,\mu)\ast_x\F_{yzt}(f(x))(\eta,\mu,\tau)\|_{L^2_{\eta\mu\tau}}
\end{align}
where $K$ is the inverse Fourier transform (in $\xi$) of the tempered distribution defined as the principal value of $|\bar{\xi}|^2/(\tau-\omega(\bar{\xi}))$. It follows that
\begin{align*}
K(\tau,x,\eta,\mu) &= \int_\R e^{ix\xi}\frac{|\bar{\xi}|^2}{\tau-\omega(\bar{\xi})}d\xi\\
&= \int_\R e^{ix(\eta^2+\mu^2)^{1/2}\xi}\frac{(\eta^2+\mu^2)^{3/2}(\xi^2+1)}{\tau-(\eta^2+\mu^2)^{3/2}\xi(\xi^2+1)}d\xi\\
&= \int_\R e^{iy\xi}\frac{\xi^2+1}{c-\xi(\xi^2+1)}d\xi
\end{align*}
with $y=(\eta^2+\mu^2)^{1/2}x$ and $c=\tau/(\eta^2+\mu^2)^{3/2}$. Next, a partial fraction expansion leads to
$$\frac{\xi^2+1}{c-\xi(\xi^2+1)} = -\frac{\alpha^2+1}{(3\alpha^2+1)(\xi-\alpha)} - \frac{\alpha}{3\alpha^2+1}\frac{2\alpha\xi+\alpha^2-1}{\xi^2+\alpha\xi+\alpha^2+1}$$
where $\alpha$ is the unique real root of $c-X(X^2+1)$. Therefore, we get
\begin{align*}
|K(\tau,x,\eta,\mu)|&\lesssim \frac{\alpha^2+1}{3\alpha^2+1}\left|\int_\R \frac{e^{iy\xi}}{\xi-\alpha}d\xi\right|
+\frac{2\alpha^2}{3\alpha^2+1}\left|\int_\R e^{iy\xi}\frac{\xi}{\xi^2+\alpha\xi+\alpha^2+1}d\xi\right|\\
&\quad +\frac{|\alpha(\alpha^2-1)|}{3\alpha^2+1}\left|\int_\R \frac{e^{iy\xi}}{\xi^2+\alpha\xi+\alpha^2+1}d\xi\right|\\
&= K_1+K_2+K_3.
\end{align*}
The first term $K_1$ is bounded by an integral witch is the Fourier transform of a function that behaves near the singular points like the kernel of the Hilbert transform $1/\xi$ (or its translates) whose Fourier transform is $\sgn(x)$. It follows that $K_1$ is bounded uniformly in $\alpha$ and $y$. In the same way, $K_2\in L^\infty$ since
\begin{align*}
K_2 &\lesssim \left|\int_\R e^{iy\xi}\frac{\xi}{\left(\xi+\frac\alpha 2\right)^2+\frac{3\alpha^2}4+1}d\xi\right|
\lesssim \frac 1{3\alpha^2+4}\left|\int_\R e^{iy\xi}\frac{\xi}{\left(\frac{\xi}{\sqrt{3\alpha^2+4}}\right)^2+1}d\xi\right|\\
&\lesssim \left|\int_\R e^{iz\xi}\frac{\xi}{\xi^2+1}d\xi\right|
\end{align*}
where $z=y\sqrt{3\alpha^2+4}$. Concerning $K_3$, it is easily estimated by
\begin{align*}
K_3 &\lesssim \frac{|\alpha(\alpha^2-1)|}{3\alpha^3+1}\int_\R\frac{d\xi}{\xi^2+\frac{3\alpha^2}{4}+1}\\
&\lesssim \frac{|\alpha(\alpha^2-1)|\sqrt{3\alpha^2+4}}{(3\alpha^2+1)(3\alpha^2+4)}\lesssim 1.
\end{align*}
We deduce that $K\in L^\infty(\R^4)$. Hence, \re{est-R11} combining with the Young inequality yields
\begin{equation}\label{est-R1}\|\nabla^2 R_1\|_{L^\infty_xL^2_{yzt}} \lesssim \|K\|_{L^\infty}\|\F_{yzt}(f(x))\|_{L^1_xL^2_{\eta\mu\tau}}\lesssim \|f\|_{L^1_xL^2_{yzt}}.\end{equation}
To treat the contribution of $R_2$, we use the smoothing bound \re{smooth} and its dual estimate \re{dual-smooth} to get
\begin{align}
\notag\|\nabla^2R_2\|_{L^\infty_xL^2_{yzt}} &\lesssim \left\|\nabla\int_\R U(-t')f(t')\sgn(t')dt'\right\|_{L^2}\\
\label{est-R2}&\lesssim \|f\|_{L^1_xL^2_{yzt}}.
\end{align}
Estimates \re{est-R1}-\re{est-R2} together with \re{dec-R} yield the desired bound.
\end{proof}

\begin{proposition}
 Let $T\le 1$, $k\ge 0$ and $\alpha\ge 3/8$. Then for any $f\in\S(\R^4)$,
\begin{align}
\label{est-nhmax0} \left\|\int_0^t U(t-t')P_0f(t')dt'\right\|_{L^2_xL^\infty_{yzT}}&\lesssim \|P_0f\|_{L^1_xL^2_{yzT}},\\
\label{est-nhmax1}\left\|t^{\alpha}\int_0^t U(t-t')\Delta_kf(t')dt'\right\|_{L^2_xL^\infty_{yzT}} &\lesssim \|\Delta_kf\|_{L^1_xL^2_{yzT}}.
\end{align}
\end{proposition}

\begin{proof}
Combining estimates \re{maxEst3d} and \re{dual-smooth} we obtain the non-retarded version of \re{est-nhmax1}:
\begin{equation}\label{est-nonre}\left\|t^{\alpha}\int_0^T U(t-t')\Delta_kf(t')dt'\right\|_{L^2_xL^\infty_{yzT}} \lesssim \|\Delta_kf\|_{L^1_xL^2_{yzT}}.\end{equation}
We consider the function $H_k(t,\bar{x}) = \left|t^{\alpha}\int_0^t U(t-t')\Delta_kf(t')dt'\right|$ that we may always assume to be continuous on $[0,T]\times\R^3$.
From \re{est-nonre} it follows that for all measurable function $t:\R^3\to [0,T]$,
$$\left\|t(x)^{\alpha}\int_0^T U(t(x)-t')\Delta_kf(t',\bar{x})dt'\right\|_{L^2_xL^\infty_{yz}} \lesssim \|\Delta_kf\|_{L^1_xL^2_{yzT}}.$$
Replacing now $f(t',\bar{x})$ by $f(t',\bar{x})\chi_{[0,t(\bar{x})]}(t')$ we see that
\begin{equation}\label{est-Hk}
\|H_k(t(\bar{x}),\bar{x})\|_{L^2_xL^\infty_{yz}}\lesssim \|\Delta_kf\|_{L^1_xL^2_{yzT}}.
\end{equation}
Since the map $t\mapsto H_k(t,\bar{x})$ is continuous on the compact set $[0,T]$, there exists $\alpha\in[0,T]$ such that $H_k(\alpha,\bar{x})=\sup_{t\in[0,T]}H_k(t,\bar{x})$, therefore
the map
$$\bar{x}\mapsto t_0(\bar{x})=\inf\left\{\alpha\in[0,T]: H_k(\alpha,\bar{x})=\sup_{t\in[0,T]}H_k(t,\bar{x})\right\}$$
is well-defined and measurable on $\R^3$. Hence, choosing $t_0$ in \re{est-Hk} we infer that
$$\|H_k(t_0(\bar{x}),\bar{x})\|_{L^2_xL^\infty_{yz}}\lesssim \|\Delta_kf\|_{L^1_xL^2_{yzT}},$$
which yields
$$\left\|\sup_{t\in[0,T]}H_k(t,\bar{x})\right\|_{L^2_xL^\infty_{yz}}\lesssim \|\Delta_kf\|_{L^1_xL^2_{yzT}}$$
and ends the proof of \re{est-nhmax1}. The proof of \re{est-nhmax0} is similar and therefore will be omitted.
\end{proof}

\section{Proofs of Theorems \re{th-besov} and \re{th-sobolev}}\label{sec-th}
In this section we solve \re{zk} in the spaces $B^{1,1}_2(\R^3)$ and $H^s(\R^3)$, $s>1$. We consider the associated integral equation
\begin{equation}
u(t) = U(t)u_0 - \frac 12\int_0^tU(t-t')\partial_x(u^2)(t')dt'.
\end{equation}
\subsection{Well-posedness in $B^{1,1}_2(\R^3)$}\label{sec-wpbesov}
For $u_0\in B^{1,1}_2(\R^3)$, we look for a solution in the space
$$X_T = \{u\in C_b([0,T], B^{1,1}_2): \|u\|_{X_T}<\infty\}$$
for some $T>0$ and where
$$\|u\|_{X_T}=N(u)+T(u)+M(u)$$
with
\begin{align*}
N(u) &= \|P_0u\|_{L^\infty_TL^2_{\bar{x}}} + \sum_{k\ge 0}2^k\|\Delta_ku\|_{L^\infty_TL^2_{\bar{x}}},\\
T(u) &= \|P_0u\|_{L^\infty_xL^2_{yzT}} + \sum_{k\ge 0}2^{2k}\|\Delta_ku\|_{L^\infty_xL^2_{yzT}},\\
M(u) &= \|P_0u\|_{L^2_xL^\infty_{yzT}} + \sum_{k\ge 0}\|t^\alpha\Delta_ku\|_{L^2_xL^\infty_{yzT}},
\end{align*}
and $3/8\le\alpha<1/2$. First from Propositions \ref{prop-smooth}, \ref{prop-max} together with the obvious bound
\begin{equation}\label{est-group}\|U(t)u_0\|_{L^\infty_TL^2_{\bar{x}}}\lesssim \|u_0\|_{L^2},\end{equation}
we get the following linear estimate:
\begin{equation}\label{est-linXT}
\|U(t)u_0\|_{X_T}\lesssim \|u_0\|_{B^{1,1}_2}.
\end{equation}
Now we need to estimate
$$\left\|\int_0^tU(t-t')\partial_x(u^2)(t')dt'\right\|_{X_T}$$
in terms of the $X_T$-norm of $u$. Using standard paraproduct rearrangements, we can rewrite $\Delta_k(u^2)$ as
\begin{align*}
 \Delta_k(u^2) &= \Delta_k\left[\lim_{j\to\infty}P_j(u)^2\right]\\
&= \Delta_k\left[P_0(u)^2+\sum_{j\ge 0}(P_{j+1}(u)^2-P_j(u)^2)\right]\\
&= \Delta_k\left[P_0(u)^2+\sum_{j\gtrsim k}\Delta_{j+1}u(P_{j+1}u+P_ju)\right].
\end{align*}
On the other hand, by similar considerations, we see that $P_0(u^2)$ can be rewritten as
$$P_0(u^2) = P_0[P_0(u)^2]+P_0\left[\sum_{j\ge 0}\Delta_{j+1}u(P_{j+1}u+P_ju)\right].$$
Hence, without loss of generality, we can restrict us to consider only terms of the form:
$$A = P_0[P_0(u)^2],\quad B = \Delta_k\left[\sum_{j\gtrsim k}\Delta_juP_ju\right], \quad C = P_0\left[\sum_{j\ge 0}\Delta_juP_ju\right]$$
for $k\ge 0$, since the estimates for the other terms would be similar.

By virtue of \re{est-nhgroup}-\re{est-nhsmooth}-\re{est-nhmax0} and \re{est-nhmax1}, we infer that
\begin{multline}\label{est-nl0}
 \left\|\int_0^tU(t-t')\partial_x(u^2)(t')dt'\right\|_{X_T} \lesssim \|P_0(u)^2\|_{L^1_xL^2_{yzT}} \\+
\sum_{k\ge 0}2^k\left(\sum_{j\gtrsim k}\|\Delta_juP_ju\|_{L^1_xL^2_{yzT}}\right)
+\sum_{j\ge 0}\|\Delta_juP_ju\|_{L^1_xL^2_{yzT}}.
\end{multline}
The first term in the right hand side is bounded by
\begin{align}
\notag\|P_0(u)^2\|_{L^1_xL^2_{yzT}} &\lesssim \|P_0u\|_{L^2}\|P_0u\|_{L^2_xL^\infty_{yzT}}\\
\notag &\lesssim T^{1/2}\|P_0u\|_{L^\infty_TL^2_{\bar{x}}}M(u)\\
\label{est-nl3} &\lesssim T^{1/2}\|u\|_{X_T}^2.
\end{align}
To evaluate the contribution of $B$, note that
\begin{align}
\notag\|\Delta_juP_ju\|_{L^1_xL^2_{yzT}} &= \|(t^{-\alpha}\Delta_ju)(t^\alpha P_ju)\|_{L^1_xL^2_{yzT}} \\
 \notag&\lesssim \|t^{-\alpha}\Delta_ju\|_{L^2}\|t^\alpha P_ju\|_{L^2_xL^\infty_{yzT}}\\
 \label{est-nl1}&\lesssim T^\mu\|\Delta_ju\|_{L^\infty_TL^2_{\bar{x}}}\|t^\alpha P_ju\|_{L^2_xL^\infty_{yzT}}
\end{align}
where $\mu = 1/2-\alpha>0$. Therefore, since
\begin{equation}\label{est-nl2}
 \|t^\alpha P_ju\|_{L^2_xL^\infty_{yzT}}\lesssim \|P_0u\|_{L^2_xL^\infty_{yzT}}+\sum_{k=0}^j\|t^\alpha \Delta_ku\|_{L^2_xL^\infty_{yzT}}\lesssim M(u),
\end{equation}
we deduce from the discrete Young inequality that
\begin{align}
\notag \sum_{k\ge 0}2^k\left(\sum_{j\gtrsim k}\|\Delta_juP_ju\|_{L^1_xL^2_{yzT}}\right) &\lesssim T^\mu M(u)\sum_{k\ge 0}\left(\sum_{j\gtrsim k}2^{k-j}(2^j\|\Delta_ju\|_{L^2_xL^\infty_{yzT}})\right)\\
\notag&\lesssim T^\mu\|u\|_{X_T}\sum_{k\ge 0}2^k\|\Delta_ku\|_{L^2_xL^\infty_{yzT}}\\
\label{est-nl4}&\lesssim T^\mu\|u\|_{X_T}^2.
\end{align}
Using again estimates \re{est-nl1}-\re{est-nl2}, we easily get that the last term in the r.h.s. of \re{est-nl0} can be estimated by
\begin{equation}\label{est-nl5}
 \sum_{j\ge 0}\|\Delta_juP_ju\|_{L^1_xL^2_{yzT}} \lesssim T^\mu M(u)\sum_{j\ge 0}\|\Delta_ju\|_{L^\infty_TL^2_{\bar{x}}}\lesssim T^\mu\|u\|_{X_T}^2.
\end{equation}
Hence, gathering \re{est-linXT}-\re{est-nl0}-\re{est-nl3}-\re{est-nl4} and \re{est-nl5}, we infer that
$$\|\mathcal{G}u\|_{X_T}\lesssim \|u_0\|_{B^{1,1}_2}+T^\mu\|u\|_{X_T}^2,$$
and in the same way,
$$\|\mathcal{G}u-\mathcal{G}v\|_{X_T}\lesssim (\|u\|_{X_T}+\|v\|_{X_T})\|u-v\|_{X_T}.$$
Hence for $T>0$ small enough, $\mathcal{G}$ is a strict contraction in some ball of $X_T$. Theorem \ref{th-besov} follows then from standard arguments.

\subsection{Well-posedness in $H^s(\R^3)$, $s>1$}
Let $u_0\in H^s(\R^3)$ with $s>1$. For $T>0$, we introduce the space
$$X_T^s = \{u\in C_b([0,T], H^s) : \|u\|_{X_T^s}<\infty\}$$
where
$$\|u\|_{X_T^s} = N(u)+T(u)+M(u)$$
with
\begin{align*}
N(u) &= \|P_0u\|_{L^\infty_TL^2_{\bar{x}}} + \left(\sum_{k\ge 0}4^{sk}\|\Delta_ku\|_{L^\infty_TL^2_{\bar{x}}}^2\right)^{1/2},\\
T(u) &= \|P_0u\|_{L^\infty_xL^2_{yzT}} + \left(\sum_{k\ge 0}4^{(s+1)k}\|\Delta_ku\|_{L^\infty_xL^2_{yzT}}^2\right)^{1/2},\\
M(u) &= \|P_0u\|_{L^2_xL^\infty_{yzT}} + \left(\sum_{k\ge 0}4^{(s-1-\eps)k}\|\Delta_ku\|_{L^2_xL^\infty_{yzT}}^2\right)^{1/2},
\end{align*}
for some $\eps>0$ small enough. From Proposition \ref{prop-smooth} together with \re{est-linhs} and \re{est-group},
$$\|U(t)u_0\|_{X_T^s}\lesssim \|u_0\|_{H^s}.$$
Following the arguments given in Subsection \ref{sec-wpbesov}, it is not too hard to see that
\begin{multline*}
 \left\|\int_0^tU(t-t')\partial_x(u^2)(t')dt'\right\|_{X_T^s} \lesssim \|P_0(u)^2\|_{L^1_xL^2_{yzT}} \\+
\left(\sum_{k\ge 0}4^{sk}\left(\sum_{j\gtrsim k}\|\Delta_juP_ju\|_{L^1_xL^2_{yzT}}\right)^2\right)^{1/2}
+\sum_{j\ge 0}\|\Delta_juP_ju\|_{L^1_xL^2_{yzT}}.
\end{multline*}
The estimates for these terms follow the same lines than the $B^{1,1}_2$ case, except that we use the Young inequality for  $\ell^1\star\ell^2$, as well as the bound
\begin{align*}
 \|P_ju\|_{L^2_xL^\infty_{yzT}}&\lesssim \|P_0u\|_{L^2_xL^\infty_{yzT}}+\sum_{k=0}^j\|\Delta_ku\|_{L^2_xL^\infty_{yzT}}\\
&\lesssim M(u)+\left(\sum_{k\ge 0}4^{(1-s+\eps)k}\right)^{1/2}M(u)\\
&\lesssim M(u),
\end{align*}
as soon as $0<\eps<s-1$. This leads to
$$\|\mathcal{G}u\|_{X_T^s}\lesssim \|u_0\|_{H^s}+T^\mu\|u\|_{X_T^s}^2,$$
and
$$\|\mathcal{G}u-\mathcal{G}v\|_{X_T^s}\lesssim (\|u\|_{X_T^s}+\|v\|_{X_T})\|u-v\|_{X_T^s}.$$
This proves the existence and uniqueness of a local solution $u$ in $X_T^s$ with $T=T(\|u_0\|_{H^s})$ small enough.

\end{document}